\documentclass{amsart}
\usepackage{fullpage}
\usepackage{amsmath}
\usepackage{mathrsfs}
\usepackage{amsfonts}
\usepackage{latexsym,epsfig}
\usepackage{mathabx}
\usepackage{amsmath, amscd, amsthm}
\usepackage[pdftex]{color}
\usepackage{tikz}
\usepackage{amscd}
\usetikzlibrary{trees,arrows}
\newcommand{\RR}[0]{\mathbb{R}}
\newcommand{\HH}[0]{\mathbb{H}}
\newcommand{\HHH}[0]{\mathcal{H}}
\newcommand{\MM}[0]{\mathcal{M}}
\newcommand{\AM}[0]{\mathcal{AM}}

\newcommand{\MCG}[0]{\mathcal{MCG}}

\newcommand{\TT}[0]{\mathcal{T}}

\newcommand{\CC}[0]{\mathcal{C}}

\newcommand{\ZZ}[0]{\mathbb{Z}}
\newcommand{\NN}[0]{\mathbb{N}}

\newcommand{\base}[0]{\mathrm{base}}
\newcommand{\tmu}[0]{\widetilde{\mu}}
\newcommand{\bmu}[0]{\widetilde{\mathbf{\mu}}}
\newcommand{\tnu}[0]{\widetilde{\nu}}
\newcommand{\bnu}[0]{\widetilde{\mathbf{\nu}}}
\newcommand{\Fmu}[0]{\mathbf{F}_{\bmu}}
\newcommand{\Fnu}[0]{\mathbf{F}_{\bnu}}

\newcommand{\Fwxx}[0]{\mathbf{F}_{\bW,\bx'}}
\newcommand{\Fwmu}[0]{\mathbf{F}_{\bW,\bmu}}

\newcommand{\bx}[0]{\widetilde{\mathbf{x}}}
\newcommand{\by}[0]{\widetilde{\mathbf{y}}}
\newcommand{\tx}[0]{\widetilde{x}}
\newcommand{\ty}[0]{\widetilde{y}}
\newcommand{\bz}[0]{\widetilde{\mathbf{z}}}
\newcommand{\bW}[0]{\mathbf{W}}
\newcommand{\bV}[0]{\mathbf{V}}
\newcommand{\tz}[0]{\widetilde{z}}
\newcommand{\Tgamma}[0]{\widetilde{\Gamma}}
\newcommand{\Tlambda}[0]{\widetilde{\Lambda}}
\newcommand{\HPhi}[0]{\widehat{\Phi}}
\newcommand{\Pwmu}[0]{P_{\bW,\bmu}}
\newcommand{\Pwxx}[0]{P_{\bW,\bx'}}
\newcommand{\ind}[0]{\text{ind}}
\newcommand{\Ind}[0]{\text{Ind}}
\newcommand{\wind}[0]{\widehat{\text{ind}}}
\newcommand{\wInd}[0]{\widehat{\text{Ind}}}
\newcommand{\wdim}[0]{\widehat{\dim}}

\usepackage{hyperref}

\newtheorem {theorem}{Theorem}[section] 
\newtheorem {lemma} [theorem] {Lemma}
\newtheorem {proposition} [theorem] {Proposition}
\newtheorem {corollary} [theorem] {Corollary}

 \newtheorem{definition}[theorem]{Definition}

\begin{document}

\title{The asymptotic geometry of the Teichm\"uller metric: Dimension and rank}
\author{Matthew Gentry Durham}
\address{Department of Mathematics, University of Michigan, 3079 East Hall, 530 Church Street, Ann Arbor, MI 48109}
\email{durhamma(at)umich.edu}
\maketitle
\begin{abstract}

We analyze the asymptotic cones of Teichm\"uller space with the Teichm\"uller metric, $(\TT(S),d_T)$.  We give a new proof of a theorem of Eskin-Masur-Rafi \cite{EMR13} which bounds the dimension of quasiisometrically embedded flats in $(\TT(S),d_T)$.  Our approach is an application of the ideas of Behrstock \cite{Beh06} and Behrstock-Minsky \cite{BM08} to the quasiisometry model we built for $(\TT(S),d_T)$ in \cite{Dur13}. 
\end{abstract}

\section{Introduction}

In this paper, we study the coarse geometry of Teichm\"uller space with the Teichm\"uller metric, $(\TT(S),d_T)$ via its asymptotic cones.   Our main goal is to bound the dimension of a quasiisometrically embedded flat $(\TT(S),d_T)$, which is called the \emph{geometric rank}.  Since quasiisometrically embedded subspaces of $(\TT(S),d_T)$ become bi-Lipschitz embedded subspaces in its cones, this involves studying bi-Lipschitz flats in the cones.\\

For $S=S_{g,n}$, the number of curves in a pants decomposition of $S$ is $r(S) = 3g-3+n$, which we call the \emph{topological dimension} of $S$.  Our main theorem is the following:

\begin{theorem}\label{r:main intro}
The maximal topological dimension of a locally compact subset of any asymptotic cone of $(\TT(S),d_T)$ is $r(S)$.
\end{theorem}

As an immediate corollary, we obtain:

\begin{corollary}\label{r:rank intro}
The geometric rank of $(\TT(S),d_T)$ is bounded above by $r(S)$.
\end{corollary}

It follows from a theorem of Minsky \cite{Min96} (see Theorem \ref{r:min prod} below) that the geometric rank is at least $r(S) - 1$.  Bowditch \cite{Bow14} has shown this bound is sharp in most cases.\\


Along the way toward the main theorem, we prove a related result about the coarse geometry of $(\TT(S),d_T)$.  In \cite{MM99}, Masur-Minsky proved that $(\TT(S),d_T)$ is weakly hyperbolic relative to its thin parts.  Introduced by Behrstock-Drutu-Mosher \cite{BDM08} , the notion of \emph{thickness} measures how far away a metric space is from being strongly relatively hyperbolic.  We prove:

\begin{theorem} \label{r:thick main}
$(\TT(S),d_T)$ is thick and thus not strongly hyperbolic relative to any collection of subspaces.
\end{theorem}

While this result is likely not surprising to the experts, it is, to our knowledge, new.

\subsection{Related results}

Corollary \ref{r:rank intro} was recently obtained by Eskin-Masur-Rafi \cite{EMR13}.  During the final stages of preparation of this work, Bowditch \cite{Bow14} proved Theorem \ref{r:main intro} and Corollary \ref{r:rank intro}.  Although all three approaches use the Masur-Minsky hierarchy machinery \cite{MM99, MM00, Raf07, Dur13} as a starting point, they are all quite different.\\

Our approach closely follows that of Behrstock \cite{Beh06} and Behrstock-Minsky \cite{BM08} for studying the asymptotic cones of the mapping class group and Teichm\"uller space with the Weil-Petersson metric.  We hope that our exposition may help further emphasize the strong connections between these three spaces.

\section{Preliminaries}

In this section, we gather the tools necessary for the exposition which follows.

\subsection{Notation}

For the remainder of the paper, let $S=S_{g,n}$ be a surface of finite type.\\

We use the following notation to help control constants: For any numbers $A$ and $B$, we write $A \prec B$ when there are uniform constants $K, C >0$ depending only on $S$ such that $A \leq K \cdot B + C$.  If $A \prec B$ and $B \prec A$, we write $A \asymp B$.
  
\subsection{$\TT(S)$ and $\MCG(S)$}

The \emph{Teichm\"uller space} of $S$, denoted $\TT(S)$, is the space of isotopy classes of marked hyperbolic (equivalently, conformal) structures on $\TT(S)$.  Teichm\"uller space admits several metrics, but we will be interested in the Teichm\"uller metric, denoted $d_T$,  which measures the quasiconformal distortion between to points in $\TT(S)$.\\

The \emph{mapping class group} of $S$, denoted $\MCG(S)$, is the group of orientation preserving homeomorphisms of $S$ modulo those isotopic to the identity.  The mapping class group acts naturally by isometries on $(\TT(S), d_T)$ by changing the marking.\\

Recently, both $\TT(S)$ and $\MCG(S)$ have been intensely studied using combinatorial machinery built from curves.  Much of rest of this section is devoted to an overview of this machinery.

\subsection{Thin parts of $\TT(S)$ are (coarsely) products}\label{prod section}

In this subsection, we recall a theorem of Minsky which characterizes the thin regions of Teichm\"uller space as product spaces.  The rest of the paper involves analyzing a coarse model of $\TT(S)$ which encodes the geometry of these product regions.\\

Let $\Gamma$ be a collection of disjoint simple closed curves on $S$.  For any $\epsilon>0$, let $\mathrm{Thin}_{\epsilon}(S, \Gamma) = \{X \in \TT(S)\hspace{.05in} | \hspace{.05in} l_X(\gamma)  < \epsilon, \forall \hspace{.05in} \gamma \in \Gamma\}$ and set

\[T_{\Gamma} = \prod_{\gamma \in \Gamma} \HH_{\gamma} \times \TT(S\setminus \Gamma)\]
where $\HH_{\gamma}$ is a copy of the upper half plane and if $S \setminus \Gamma = \coprod Y$ is disconnected, then we take $\TT(S\setminus \Gamma) = \prod \TT(Y)$.  We put the Teichm\"uller metric on each component of $\TT(S\setminus \Gamma)$ and consider $T_{\Gamma}$ with the sup metric.  Minsky proved:

\begin{theorem}[Minsky \cite{Min96}]\label{r:min prod}
There are $\epsilon, C>0$ such that Fenchel-Nielsen coordinates give rise to a natural homeomorphism $\Phi: \TT(S) \rightarrow T_{\Gamma}$ which restricts to a $(1,C)$-quasiisometric embedding on $\mathrm{Thin}_{\epsilon}(S,\Gamma)$.
\end{theorem}

One can use Minsky's theorem to give a lower bound on the geometric rank of $\TT(S)$ of $r(S)-1$ as follows: Let $P \subset \CC(S)$ be a maximal simplex and consider its corresponding thin part in $\TT(S)$.  Take a geodesic ray in each horodisk component of the corresponding product region.  The product of these rays gives an $r(S)$-dimensional quasiorthant in $\TT(S)$, whose boundary is a quasiflat of dimension $r(S)-1$.

\subsection{Curves, subsurfaces, and markings}

The \emph{curve complex} of $S$, denoted $\CC(S)$, is a graph whose vertices are curves with edges usually for disjointness.  When $S=S_{1,1},$ or $S_{0,4}$, there are edges for minimal intersection and $\CC(S)$ is a Farey graph.  In the case of $Y_{\alpha}$, the annular collar of some curve $\alpha \in \CC(S)$, we require a further refinement: Let $\widetilde{Y}$ be the cover of $S$ corresponding to $Y_{\alpha}$; $\widetilde{Y}$ has a natural compactification to a closed annulus, $\widetilde{Y}'$, and we let the vertices of $\CC(Y_{\alpha})$ be the paths which connect the boundary components of $\widetilde{Y}'$, relative to homotopies which fix the endpoints; edges are again for disjointness.  In this case, $\CC(Y_{\alpha})= \CC(\alpha)$ is quasiisometric to $\ZZ$.\\

The following theorem is the core of the coarse approach:

\begin{theorem}[\cite{MM99}]
$\CC(S)$ is an infinite diameter, Gromov hyperbolic space.
\end{theorem}

Given a simplex $\gamma \subset \CC(S)$, if the complement $S \setminus \gamma$ is disconnected, then we call the components \emph{subsurfaces}.  For any subsurface $Y \subset S$, we denote the disjoint collection of curves which bound it by $\partial Y \subset \CC(S)$.\\

In \cite{MM00}, Masur-Minsky built a quasiisometry model for $\MCG(S)$ called the \emph{marking complex}, $\MM(S)$.  A (complete) \emph{marking} $\mu \in \MM(S)$ on $S$ is collection of \emph{transverse pairs} $(\alpha, t_{\alpha})$, where the $\alpha$ is a pants decomposition called the \emph{base of} $\mu$, denoted by $\mathrm{base}(\mu)$, and each $t_{\alpha}$, called the \emph{transversal} to $\alpha$, is a curve intersecting $\alpha$ such that the subsurface, $Y$, filled by $\alpha \cup t_{\alpha}$ satisfies $r(Y)=1$ and $d_Y(\alpha, t_{\alpha}) = 1$ (see Section 2.4 of \cite{MM00}).  In addition, all our markings are \emph{clean}: the only base curve which each transversal intersects is its paired base curve.\\

Two markings are connected by an edge in $\MM(S)$ if they differ by one of two elementary moves: 
\begin{enumerate}
\item[(Twist move)] A Dehn twist or half twist around a base curve
\item[(Flip move)] Switch the roles of a base curve and its transverse curve$(\alpha, t_{\alpha}) \mapsto (t_{\alpha}, \alpha)$ (along with some other coarsely inconsequential changes to make the resulting marking clean).
\end{enumerate}

Since $\MM(S)$ has finite valence and $\MCG(S)$ acts properly and cocompactly on it, we have :

\begin{lemma}[\cite{MM00}]\label{marking qi}
The marking complex $\MM(S)$ with the graph metric is quasiisometric to $\MCG(S)$ with any word metric.
\end{lemma}

We also need the notion of a \emph{subsurface projection}, which records the combinatorics of two curves from the perspective of a subsurface.\\

Let $\alpha \in \CC(S)$ be any simplex and let $Y \subset S$ be any subsurface with $r(Y)\neq 0$.  The \emph{subsurface projection} of $\alpha$ to $Y$, denoted $\pi_Y(\alpha)\subset \CC(Y)$, is the completion of $\alpha \cap Y$ along the boundary of a regular neighborhood of $\alpha \cap Y$ and $\partial Y$ to curves in $Y$.  If $Y = Y_{\gamma}$ is an annulus with core $\gamma$, then $\pi_{Y_{\gamma}}(\alpha) = \pi_{\gamma}(\alpha)$ is the set of lifts of $\gamma$ to the annular cover $\widetilde{Y}_{\gamma}$ of $S$ which connect the two boundaries of the compactification of $\widetilde{Y}_{\gamma}$.  In both cases $\pi_Y(\alpha) \subset \CC(Y)$ is a simplex, unless $\alpha \cap Y = \emptyset$ and then $\pi_{Y}(\alpha) = \emptyset$.\\

For $\mu \in \MM(S)$ and $Y$ nonannular, we set $\pi_Y(\mu) = \pi_Y(\mathrm{base}(\mu))$.  If $Y = Y_{\alpha}$ is an annulus with core $\alpha \in \mathrm{base}(\mu)$ and transversal $t_{\alpha}$, then $\pi_{\alpha}(\mu) = t_{\alpha}$.  See Section 2 of \cite{MM00} for more details.\\

When measuring the distance between the projection of two curves or markings to a subsurface, we typically write $d_Y(\pi_Y(\mu_1), \pi_Y(\mu_2)) = d_Y(\mu_1,\mu_2)$.  We also have:

For $\mu \in \MM(S)$ and subsurface $Y \subset S$, we build the \emph{projection of $\mu$ to a marking on $\MM(Y)$}, $\pi_{\MM(Y)}(\mu)$, by induction as follows.  Choose a curve $\gamma_1 \in \pi_Y(\mu)$ and build a pants decomposition on $Y$ by choosing $\gamma_i \in \pi_{Y \setminus \bigcup_{j=1}^{i-1} \gamma_j} (\mu)$.  Using this pants decomposition as its base, build a marking on $Y$ by choosing transverse pairs $(\gamma_i, \pi_{\gamma_i}(\mu))$.  Define $\pi_{\MM(Y)}(\mu) \subset \MM(Y)$ to be the collection of all markings resulting from varying the choices of the $\gamma_i$.\\

\cite{MM00}[Lemma 2.4] and  \cite{Beh06}[Lemma 6.1] show that this projection is coarsely well-defined.  We remark that if $\partial Y \subset \mathrm{base}(\mu)$, then $\pi_{\MM(Y)}(\mu)$ is a unique point in $\MM(Y)$, since every curve in $\mathrm{base}(\mu)$ is either contained in or disjoint from $Y$.  In fact, one can show:

\begin{lemma}[Lipschitz projection; \cite{MM00}, \cite{Dur14}] \label{lip proj}
Let $X \subset Y \subset S$ be subsurfaces.  For any augmented marking $\tmu \in \AM(Y)$, if $\pi_X(\tmu) \neq \emptyset$, then $\text{diam}_{\AM(X)}(\tmu) \asymp 1$.
\end{lemma}

\subsection{The augmented marking complex}

In this subsection, we review our main construction from \cite{Dur13}, as further refined in \cite{Dur14}.\\

The \emph{augmented marking complex} of $S$, $\AM(S)$, is a graph whose vertices are augmented markings.  An \emph{augmented marking} $\tmu$ is a collection of data $\left(\mu, \{D_{\alpha}\}_{\alpha \in \base(\mu)}\right)$, where $\mu \in \MM(S)$ and each $D_{\alpha} \in \ZZ_{\geq 0}$ is a \emph{coarse length coordinate} which specifies coarsely how short each base curve is.\\ 

Two augmented markings $\tmu_1, \tmu_2 \in \AM(S)$ are connected by an edge in $\AM(S)$ if they differ by one of the following types of elementary moves which extend those of $\MM(S)$:

\begin{itemize}
\item[(Flip moves)] If  $\mu_1, \mu_2 \in \MM(S)$ differ by a flip move on a transverse pairing $(\alpha, t_{\alpha}) \mapsto (t_{\alpha}, \alpha)$, and if $\tilde{\mu}_1, \tilde{\mu}_2$ have the same base curves and length data, with $D_{\alpha}(\tilde{\mu}_1) = D_{\alpha}(\tilde{\mu}_2) = 0$ for each $\alpha \in \mathrm{base}(\tilde{\mu}_1) = \mathrm{base}(\tilde{\mu}_2)$.
\item[(Twist moves)] If $\alpha \in \mathrm{base}(\mu_1) = \mathrm{base}(\mu_2)$, $D_{\alpha}(\tilde{\mu}_1) = D_{\alpha}(\tilde{\mu}_2) = k > 0$, and $\tilde{\mu}_1= T^n_{\alpha} \tilde{\mu}_2$ with $0 < n < e^k$, where $T_{\alpha}$ is the positive Dehn (half)twist around $\alpha$.
\item[(Vertical moves)] If $\mu_1 = \mu_2$ and there is an $\alpha \in \mathrm{base}(\mu_1) = \mathrm{base}(\mu_2)$ such that $D_{\alpha}(\tilde{\mu}_1) = D_{\alpha}(\tilde{\mu}_2) \pm 1$ and $D_{\beta}(\tilde{\mu}_1) = D_{\beta}(\tilde{\mu}_2)$ for all $\beta \in \mathrm{base}(\mu_1) \setminus \alpha = \mathrm{base}(\mu_2) \setminus \alpha$.
\end{itemize}

Note that its not possible to perform a flip move at $\alpha \in \base(\tmu)$ if $D_{\alpha}(\tmu)>0$.  For any $\tmu \in \AM(S)$ and $\alpha \notin \base(\tmu)$, we define $D_{\alpha}(\tmu) =0$.  Note that $\{\tmu \hspace{.05in} | \hspace{.05in} D_{\alpha}(\tmu)=0, \forall \alpha \in \CC(S)\}$ is a metrically distorted copy of $\MM(S)$ at the base of $\AM(S)$.\\

The following was the main theorem of \cite{Dur13}:

\begin{theorem} \label{r:aug qi}
The augmented marking complex, $\AM(S)$, is $\MCG(S)$-equivariantly quasiisometric to $\TT(S)$ with the Teichm\"uller metric.
\end{theorem}

Let $\tmu \in \AM(S)$ and $\alpha \in \base(\tmu)$.  There is a special subgraph of $\AM(S)$ consisting of augmented markings which differ from $\tmu$ by twist and vertical moves at $\alpha$ called a \emph{combinatorial horoball}.  These horoballs are the $\AM(S)$-analogues of annular curve complexes for $\MM(S)$.\\

More generally, a \emph{combinatorial horoball over $\ZZ$}, $\mathcal{H}(\ZZ)$, is the 1-complex with vertices $\mathcal{H}(\ZZ) = \ZZ \times (\{0\} \cup \NN)$ and edges as follows:
\begin{itemize}
\item If $x,y \in \ZZ$ and $m\in \{0\} \cup \NN$ such that $0<|x-y| \leq 2^m$, then $(x,m)$ and $(y,m)$ are connected by an edge in $\mathcal{H}(\ZZ)$.
\item If $x \in \ZZ$ and $m\in \{0\} \cup \NN$, then $(x,m)$ is connected to $(x,m+1)$ by an edge.
\end{itemize}

\begin{lemma}[\cite{Dur13}]\label{r:horoball qi}
The combinatorial horoball over $\ZZ$, $\HHH(\ZZ)$, is quasiisometric to the horodisk $\HH^2_{\geq 1}$.
\end{lemma}

Combinatorial horoballs in $\AM(S)$ stand in for the horodisks appearing in Minsky's product regions Theorem \ref{r:min prod} and are the $\AM(S)$-analogues of annular curve complexes from $\MM(S)$.  As such, we want to compare augmented markings on combinatorial horoballs.  Doing so requires some care, as annular curve complexes are only quasiisometric to $\ZZ$.  We recall some notation from [Subsection 4.2, \cite{Dur13}].\\

For each $\alpha \in \CC(S)$, choose an arc $\beta_{\alpha} \in \CC(\alpha)$.  For any other $\gamma \in \CC(\alpha)$, let $\gamma \cdot \beta_{\alpha}$ denote the algebraic intersection number.  The map $\phi_{\beta_{\alpha}}: \CC(\alpha) \rightarrow \ZZ$, given by $\phi_{\beta_{\alpha}}(\gamma) = \gamma \cdot \beta_{\alpha}$, is a $(1,2)$-quasiisometry (independent of $\beta_{\alpha}$) which records the twisting of $\gamma$ around $\alpha$ relative to $\beta_{\alpha}$.\\

Let $\HHH_{\alpha} = \HHH(\ZZ)$ be the combinatorial horoball over $\ZZ$.  Define $\pi_{\HHH_{\alpha}}:\AM(S) \rightarrow \HHH_{\alpha}$ as follows: For any $\tilde{\mu} \in \AM(S)$,
\begin{equation*} \pi_{\HHH_{\alpha}}(\tilde{\mu}) = \left\{
\begin{array}{lr} \left(\phi_{\beta_{\alpha}}(t_{\alpha}), D_{\alpha}\right) & \text{if } (\alpha, t_{\alpha}, D_{\alpha}) \in \tilde{\mu}\\
\left(\phi_{\beta_{\alpha}}(\pi_{\alpha}(\tilde{\mu})),0\right) & \text{otherwise}
\end{array}
\right.
\end{equation*}

We note that any error coming from a choice of $\beta_{\alpha} \in \CC(\alpha)$ is uniformly bounded.\\

We need to understand how to project an augmented marking to an augmented marking on a subsurface.\\

For any augmented marking $\tilde{\mu} \in \AM(S)$ and nonannular subsurface $Y \subset S$, the \emph{projection of $\tilde{\mu}$ to $\AM(Y)$} by setting $\pi_{\MM(Y)}(\mu)$ to be the underlying marking of $\pi_{\AM(Y)}(\tilde{\mu})$ and, for each $\alpha \in \mathrm{base}(\pi_{\MM(Y)}(\mu))$, we set $D_{\alpha}(\pi_{\AM(Y)}(\tilde{\mu}))$ equal to $D_{\alpha}(\tilde{\mu})$ if $\alpha \in \base(\tilde{\mu})$ and 0 otherwise.  In the case that $Y \subset S$ is an annulus with core curve $\beta$, then $\pi_{\AM(Y)}(\tilde{\mu}) = \widehat{\pi}_{\HHH_{\beta}}(\tilde{\mu})$.

\subsection{Hierarchy paths and distance formulae}

The hierarchy machinery of Masur-Minsky \cite{MM00} is a powerful, though technical, tool for understanding the coarse geometry of $\MCG(S)$.  It has two main outputs: a coarse distance formula for $\MM(S)$ (and thus $\MCG(S)$) in terms of subsurface projections and families of uniform quasigeodesic paths between any pair of markings, called \emph{hierarchy paths}.  In  \cite{Raf07}, Rafi showed that markings can coarsely encode much about a point of $\TT(S)$ in the Teichmu\"uller metric and he obtained a coarse distance formula for the Teichm\"uller metric partially in terms of subsurface projections.  In \cite{Dur13} we completed this analogy with augmented markings and adapted this distance formula to $\AM(S)$.  Along the way, we built families of uniform quasigeodesic path in $\AM(S)$ from hierarchy paths, which we call \emph{augmented hierarchy paths}.  In this subsection, we briefly collect the results we need; see \cite{MM00}, \cite{Raf07}, and \cite{Dur13} for more details.\\  

The following theorem collects some of the basic properties of augmented hierarchy paths:

\begin{theorem}[Theorem 4.2.3, Proposition 4.3.5 in \cite{Dur13}] \label{ahp thm}
Let $\tmu_1, \tmu_2 \in \AM(S)$ be any pair of augmented markings and let $\mu_1, \mu_2 \in \MM(S)$ be their underlying markings.  Let $H$ be any hierarchy between $\mu_1$ and $\mu_2$ with base geodesic $g_H \subset \CC(S)$, and $\Gamma \subset \MM(S)$ a hierarchy path based on $H$.  Then there exists an augmented hierarchy path $\widetilde{\Gamma} \subset \AM(S)$ based on $H$ between $\tmu_1$ and $\tmu_2$ with the following properties:

\begin{enumerate}
\item $\widetilde{\Gamma}$ is a uniform quasigeodesic
\item $\widetilde{\Gamma}$ has $\Gamma$ as its shadow in $\MM(S)$
\item The shadow of $\widetilde{\Gamma}$ in any curve complex or horoball is an unparametrized quasigeodesic
\end{enumerate}

\end{theorem}

Hierarchy paths encode subsurface projection data, about which we recall some relevant results.\\

The following lemma, which follows from [Lemma 6.2, \cite{MM00}] and [Theorem 4.2.3, \cite{Dur13}], explains which subsurfaces are always involved:

\begin{lemma}[Large links]\label{large link}
There is a $K_1>0$ such that following holds.  Let $\widetilde{\Gamma}$ be any augmented hierarchy path between $\tmu_1, \tmu_2 \in \AM(S)$ based on a hierarchy $H$.  Suppose that $d_Y(\tmu_1, \tmu_2)>K_1$, where $d_Y = d_{\HHH_{\alpha}}$ if $Y=Y_{\alpha}$.  Then exist a vertex $\tmu' \in \Tgamma$ with $\partial Y \subset \base(\tmu')$ and a vertex $v \in g_H$ with $Y \subset S \setminus v$. 
\end{lemma}

The next theorem explains the large links terminology.  Masur-Minsky originally proved it holds for projections to curve complexes, but it holds for projections to horoballs by definition:

\begin{theorem}[Bounded geodesic image theorem; \cite{MM00}]\label{bgit}
There is a constant $M_0>0$ such that the following holds.  Let $\gamma \subset \CC(S)$ be any geodesic and $Y \subset S$ any subsurface.  If $d_{\CC(S)}(\gamma, \partial Y) > 1$, then $diam_{\CC(Y)}(\gamma) < M_0$.

\end{theorem}

\begin{proof}
The only case at issue is when $Y= \HHH_{\alpha}$ for some $\alpha \in \CC(S)$.  Since $d_{\CC(S)}(\gamma, \alpha) > 1$, $D_{\alpha}(\gamma_i) = 0$ for each $\gamma_i \in \gamma$ and so $\mathrm{diam}_{\HHH_{\alpha}} (\gamma) \asymp \log \mathrm{diam}_{\CC(\alpha)}(\gamma)< \mathrm{diam}_{\CC(\alpha)}(\gamma)$.
\end{proof}

The following theorem combines the distance formulae from \cite{MM00} and \cite{Raf07}, respectively, and says that distances in $\MM(S)$ and $\AM(S)$ are coarsely determined by projections to large links:

\begin{theorem}\label{distance}
There exists a $K'>0$ so that for any $K>K'$ and any augmented markings $\tmu_1, \tmu_2 \in \AM(S)$ with underlying markings $\mu_1, \mu_2 \in \MM(S)$, the following hold:

\begin{itemize}
\item From \cite{MM00}, we have:

$$d_{\MM(S)}(\mu_1, \mu_2) \asymp_K \sum_{Y \subset S} \left[\left[d_Y(\mu_1, \mu_2)\right]\right]_K+ \sum_{\alpha \in \CC(S)} \left[\left[d_{\alpha}(\mu_1, \mu_2)\right]\right]_K$$

\item From \cite{Raf07}, as recorded in \cite{EMR13} and \cite{Dur13}, we have:

$$d_{\AM(S)}(\tmu_1, \tmu_2) \asymp_K \sum_{Y \subset S} \left[\left[d_Y(\mu_1, \mu_2)\right]\right]_K + \sum_{\alpha \in \CC(S)} \left[\left[d_{\HHH_{\alpha}}(\tmu_1, \tmu_2)\right]\right]_K$$

\end{itemize}
where the $Y \subset S$ are taken to be nonannular and $[[x]]_K = x$ if $x>K$ and $0$ otherwise.
\end{theorem}

We say a two subsurfaces $X$ and $Y$ \emph{interlock}, and write $X \pitchfork Y$, if $X\cap Y \neq \emptyset$ and neither is properly contained in the other.\\

The following is a result of Behrstock \cite{Beh06}.  It roughly states that if two subsurfaces interlock, then any augmented marking is close to at least one of the subsurfaces from the perspective of the other.  The $\AM(S)$ version immediately follows from the fact that intersecting curves cannot be simultaneously short:

\begin{proposition}[Behrstock's inequality] \label{beh ineq}
There is a constant $M_1>0$ such that the following holds.  Let $Y, Z \subset S$ be proper subsurfaces such that $Y \pitchfork Z$.  Then for any augmented marking $\tmu \in \AM(S)$, we have

$$\min \{d_Y(\tmu, \partial Z), d_Z(\tmu, \partial Y) \} < M_1$$

\end{proposition}

We want to understand when an augmented hierarchy path makes progress through a subsurface.  We recall the notion of an active segment of subsurface along an augmented hierarchy path from \cite{Dur14}.  Let $\tmu_1, \tmu_2 \in \AM(S)$, $\Tgamma$ an augmented hierarchy path between them, and let $Y \subset S$ be any subsurface.  The \emph{active segment} for $Y$ along $\Tgamma$ is the (possibly empty) segment $\Tgamma_Y \subset \Tgamma$ such $\partial Y \subset \base(\tmu)$ for each augmented marking $\tmu \in \Tgamma_Y$; Lemma 5.6 in \cite{Min03} says that $\Tgamma_Y$ is connected.  The following lemma says that an augmented hierarchy path coarsely only makes progress along a subsurface during its active segment:

\begin{lemma}[Active segments; \cite{Dur14}] \label{r:active segment}
Let $\Tgamma$ be as above.  Suppose that $Y \subset S$ has nonempty active segment, $\Tgamma_Y$.  There is an $M_2>0$ depending only on $S$ such that the following hold:
\begin{enumerate}

\item For any $\tilde{\eta}_1, \tilde{\eta}_2 \in \Tgamma$ preceding and following $\Tgamma_Y$, respectively, we have
\[d_Y(\tilde{\mu}_1, \tilde{\eta}_1), d_Y(\tilde{\mu}_2, \tilde{\eta}_2)< M_2\]

\item If $Z\subset S$ is any subsurface with $Y \pitchfork Z$ and $d_Y(\tmu_1, \partial Z) < M_1$, then $d_Z(\tmu_2, \partial Y) < M_1$ and 
\[\mathrm{diam}_Y(\Tgamma_Z),d_Y(\tilde{\mu}_1,\Tgamma_Z), \mathrm{diam}_Z(\Tgamma_Y),d_Z(\tilde{\mu}_2, \Tgamma_Y)<M_2\] \label{r:active interval 2}

\item Moreover, if $\partial Y \cap \partial Z \neq \emptyset$, then $\Tgamma_Y \cap \Tgamma_Z = \emptyset$.

\end{enumerate}

\end{lemma} 

Finally, we record the following lemma which follows easily from work in \cite{MM00} and \cite{Dur13}:

\begin{lemma}\label{dist rel to endpoints}
Let $[\tx, \ty]$ be an augmented hierarchy path, $\tz \in [\tx, \ty]$, and $Y \subset S$.  Then $d_{\AM(Y)}(\tx, \tz) \prec d_{\AM(Y)}(\tx, \ty)$.
\end{lemma}

\subsection{Product regions in $\AM(S)$} \label{prod sect}

There are natural product regions in $\AM(S)$ which correspond to the Minsky product regions from Theorem \ref{r:min prod}.  We recall some facts from \cite{Dur14}.\\

For any simplex $\Delta \subset \CC(S)$, set $Q(\Delta) = \{\tmu \hspace{.05in}|\hspace{.05in} \Delta \subset \base(\tmu)\}$.  Let $\sigma(\Gamma)$ be all nonpants subsurfaces of $S \setminus \Delta$, including the annuli around each $\gamma \in \Delta$.  The following lemmata are easy consequences of Theorem \ref{distance}:

\begin{lemma}\label{prod reg}
There is a natural quasiisometry defined by subsurface projections:

\[\Theta: Q(\Delta) \rightarrow \prod_{Y \in \sigma(\Delta)} \AM(Y)\]
\end{lemma}

We write $Y \pitchfork \Gamma$ if $\Gamma$ cannot be deformed away from $Y$.

\begin{lemma}\label{dist to prod}
For $\tx, \ty \in Q(\Delta)$, we have $d_Y(\tx, \ty) \asymp 1$ for any $Y \pitchfork \Gamma$ and thus

\[d_{\AM(S)}(\tx, \ty) \asymp \sum_{Y \subset \sigma(\Gamma) } \left[d_Y(\tx, \ty)\right]_K\]
\end{lemma}

\subsection{Asymptotic cones}

The asymptotic cone of a metric space encodes its geometry as seen from arbitrarily far away.  Originally introduced by Gromov \cite{Gro81} (see also \cite{DW84}) to prove his famous polynomial growth theorem, asymptotic cones have been widely used to study the large scale geometry of many groups and spaces.  We now give a brief overview of the basic notions.\\

A \emph{(nonprincipal) ultrafilter} is a finitely additive probability measure $\omega: 2^{\mathbb{N}} \rightarrow \{0,1\}$ for which every finite set has measure zero.  The existence of such a measure is an easy consequence of Zorn's Lemma.\\

Let $X$ be a geodesic metric space.  Given a sequence $\{x_n\} \subset X$ and a point $x$, we say that $x$ is an $\omega$-\emph{ultralimit} of $\{x_n\}$ if every open neighborhood $U$ of $x$ satisfies $\omega\{n\hspace{.05in}|\hspace{.05in}x_n \in U\} = 1$.  We write $\lim_{\omega} x_n = x$ or $x_n \underset{\omega}{\longrightarrow} x$.  We note that ultralimits are unique when they exist and every sequence in a compact set has an ultralimit.\\

We define the ultralimit of a based sequence of metric spaces $(X_n, x_n, dist_n)$ as follows.  For any sequence $\by = (y_n) \in \prod_n X_n$, define $d(\bx, \bz) = \lim_{\omega} dist_n(y_n, z_n)$.  We can define
$$(X_n, x_n, dist_n) = \{\by \hspace{.05in}| \hspace{.05in}d(\bx, \by) < \infty\}/\sim$$
where $\by \sim \bz$ if $d(\by, \bz) = 0$, making the quotient a metric space.\\

For the rest of this paper, fix a nonprincipal ultrafilter, $\omega$, and a sequence of numbers, $s_n \rightarrow \infty$.  Pick a base point sequence $\{x_n\} \subset X$.  The \emph{asymptotic cone} of $(X, dist)$ is the ultralimit
$$\mathrm{Cone}_{\omega}(X, x_n, dist) = \lim_{\omega} (X, x_n, \frac{dist}{s_n})$$

We note that $\mathrm{Cone}_{\omega}(X, x_n, dist)$ is a geodesic space since $X$ is.\\

We now make some brief observations about sequences of subsurfaces.  Any sequence of subsurfaces $\bW = \{W_n\}$ has only finitely-many topological types and thus this type is $\omega$-a.e. constant.  We call this type the \emph{topological type} of $\bW$.  Given another sequence $\bV=\{V_n\}$ of subsurfaces, we can similarly define $\bV \subset \bW$ or $\bV \pitchfork \bW$ if $V_n \subset W_n$ or $V_n \pitchfork W_n$ for $\omega$-a.e. $n$.

\section{Thick and thin cones}

The Cayley graph of a finitely generated group, such as $\MCG(S)$, is homogeneous  in the sense that every vertex looks like the identity up to the action of the group.  One well-known consequence is that the asymptotic cone of a finitely generated group is independent of the choice of base point sequence.  By contrast, $\AM(S)$ does not even have bounded valence.  In this section, we classify the asymptotic cones of $\TT(S)$ up to bi-Lipschitz homeomorphism by analyzing how a cone depends on the choice of basepoint sequence.
\\

The choice of base point sequence essentially breaks into two collections, which we call \emph{thick} or \emph{thin}, depending on whether or not the sequence escapes the thick part faster than the scaling sequence.  In the thick case, we prove that all corresponding asymptotic cones are bi-Lipschitz homeomorphic.  By contrast, the thin case breaks into the asymptotic cones of the Minsky product regions.  The bulk of this section deals with analyzing the possible thin cases.\\

Let $\bmu = \{\tmu_n\} \subset \AM(S)$ be a sequence of augmented markings and let $\base(\tmu_n) = \{\alpha_{n,i}\}_i$ be the corresponding bases.  For each $n$, reorder the base curves from largest coarse length to smallest.  Let

$$\bold{A}(\bmu) = \left|\left\{i \hspace{.05in} \Big| \hspace{.05in}\frac{D_{\alpha_{n,i}}(\tmu_n)}{s_n}\underset{\omega}{\longrightarrow} \infty \right\}\right|$$
be the number of base curves in $\tmu_n$, for $\omega$-a.e. $n$, whose lengths escape faster to infinity than $s_n$.\\

When $\bold{A}(\bmu) = 0$, we say that $\bmu$ is $\emph{thick}$ and $\bmu$ is \emph{thin} otherwise.  When $\bmu$ is thin, then $\tmu_n$ is more than $\bold{A}(\bmu)\cdot s_n$ away from the thick part for $\omega$-a.e. $n$.  That is, $\bmu$ is escaping the thick part of $\AM(S)$ faster than $s_n$.  We emphasize that this does not imply that the coarse length coordinates of the augmented markings in a thick base sequence cannot escape to infinity, just that they do so in a controlled manner.

\subsection{Thick cones}

The following proposition states that all asymptotic cones of $\AM(S)$ with thick base sequences are bi-Lipschitz homeomorphic:

\begin{proposition}[Thick asymptotic cones are all coarsely the same]
Let $\bmu, \bnu \subset \AM(S)$ be thick sequences.  Then $(\AM_{\omega}(S),\bmu)$ and $(\AM_{\omega}(S),\bnu)$ are bi-Lipschitz homeomorphic.
\end{proposition}

\begin{proof}
For each $n$, let $\mu_n, \nu_n \in \MM(S)$ be the markings underlying $\tmu_n, \tnu_n \in \AM(S)$, respectively.  Since $\MCG(S)$ acts coarsely transitively by isometries on $\MM(S)$, there are a constant $D>0$ and a sequence of mapping classes $\phi_n \in \MCG(S)$ with $d_{\MM(S)}(\phi_n(\mu_n), \nu_n) \leq D$, for each $n$.  We claim that there is a $D'>0$ with

$$\frac{d_{\AM(S)}(\phi_n(\tmu_n),\tnu_n)}{s_n} \underset{\omega}{\longrightarrow} D'$$

Set $\widehat{\phi} = \lim_{\omega} \phi_n$.  The claim implies that $\widehat{\phi}: (\AM(S), \bmu) \rightarrow (\AM(S), \bnu)$ is a bi-Lipschitz homeomorphism.\\

To see the claim, observe that Theorem \ref{distance} implies that

\begin{eqnarray*}
d_{\AM(S)}(\phi_n(\tmu_n),\tnu_n) &\prec& d_{\MM(S)}(\phi_n(\mu_n),\nu_n) + \sum_{\alpha} \left|D_{\alpha}(\phi_n(\tmu_n)) - D_{\alpha}(\tnu_n)\right|\\
&\leq& d_{\MM(S)}(\phi_n(\mu_n), \nu_n) + \sum_{\alpha} \left(D_{\phi_n(\alpha)}(\tmu_n) + D_{\alpha}(\tnu_n)\right)
\end{eqnarray*}

Note that both sums above always have finitely many terms.  Since  $\bmu$ and $\bnu$ are thick and $d_{\MM(S)}(\phi_n(\mu_n), \nu_n)$ is uniformly bounded, the claim follows.

\end{proof}

Much of the rest of the paper is proving the Dimension Theorem for thick base sequences.  The remainder of this section is occupied with proving that the thin case can be reduced to the thick case.  The Dimension Theorem for the thin case is easily reduced to the thick case in Corollary \ref{dim thm 3} by induction.

\subsection{Thin cones are products} \label{r:thin cone section}

For the rest of this section, suppose that $\bmu$ is thin, i.e. that $\mathbf{A}(\bmu) > 0$.\\

For each $n$, let 

$$X_n(\tmu_n) = S \hspace{.05in} \setminus \hspace{.05in} \underset{i \leq \mathbf{A}(\bmu)}{\coprod} \alpha_{n,i}$$

Set $\bold{X} (\bmu)= \{X_n(\tmu_n)\}_n$.  Each $X_n(\tmu_n)$ is a disjoint union of subsurfaces and, because there are only finitely many, the topological type of $X_n$ is $\omega$-a.e. constant.  We call this number the topological type of $\bold{X}(\bmu)$.\\

Let $\bold{X}(\bmu) = \coprod_j \mathbf{Y}_j$ be the components of $\bold{X}(\bmu)$, where we ignore any components which are pairs of pants.  To be precise, for $\omega$-a.e. $n$ and for each $j$, there is some $Y_{n,j} \subset X_n(\tmu_n)$ with $Y_{n,j}$ homeomorphic to $\mathbf{Y}_j$.\\

By passing to a subsequence with indices of full $\omega$-measure, we may that for each $n$:

\begin{itemize}
\item For each $i= 1, \dots, \bold{A}(\bmu)$, we have $D_{\alpha_{n,i}}(\tmu_n)>s_n$ and \item For each $j$, there is a $Y_{n,j} \subset X_n(\tmu_n)$ with $Y_{n,j}$ homeomorphic to $\mathbf{Y}_j$
\end{itemize}

The following theorem states that the asymptotic cone of $\TT(S)$ with a thin base sequence is the product of the asymptotic cones of the Teichm\"uller spaces of the components, $Y_j$, and some $\RR$-trees which are the asymptotic cones of the horoballs over the thin curves, $\alpha_i$.\\

Before we state the theorem, we choose natural base sequences for each component.\\

For each  $i= 1, \dots, \bold{A}(\bmu)$, let $\mathbf{x}_i= \pi_{\HHH_{\alpha_{n,i}}}(\tmu_n)$, where the horoball projection $\pi_{\HHH_{\alpha_{n,i}}}$ is taken relative to the transversal of $\alpha_{n,i}$ in $\tmu_n$.  Let $\bold{T}_i = \underset{\omega}{\lim}\left(\HHH_{\alpha_{n,i}}, \mathbf{x}_i,s_n\right)$, which is an $\RR$-tree.\\

For each $j$, let $\tmu_{n,j} = \pi_{\AM(Y_{n,j})}(\tmu_n)$.  Set $\AM_{\omega}(\mathbf{Y}_j) = \underset{\omega}{\lim}\left(\AM(Y_{n,j}),\tmu_{n,j},s_n\right)$.  Note that $\tmu_{n,j}$ is a thick sequence relative to $\AM(Y_{n,j})$.

\begin{theorem}[Classification of thin asymptotic cones] \label{thin cones}
Let $\bmu \subset \AM(S)$ be such that $\bold{A}(\bmu) > 0$.  Then the asymptotic cone of $(\AM(S),\bmu, s_n)$ is bi-Lipschitz homeomorphic to

$$\prod_j \AM_{\omega}(Y_i) \times \prod_{i=1}^{\bold{A}(\bmu)} \bold{T}_i$$

\end{theorem}

\begin{proof}
For this proof, set $\AM_{\omega}(S) = \lim_{\omega}\left(\AM(S), \bmu, s_n\right)$.  Set $\Delta_n = \coprod_{1\leq i \leq \bold{A} (\bmu)} \alpha_{n,i}$.  Note that for each $n$, we have $\tmu_n \in Q(\Delta_n)$.\\

Let $\bnu = \{\tnu_n\}_n \subset \AM_{\omega}(S)$ be arbitrary.  We will show that $\tnu_n \in Q(\Delta_n)$ for $\omega$-a.e. $n$.  Since taking asymptotic cones commutes with products, Proposition \ref{prod reg} will then imply the result.\\

To show that $\tnu_n \in Q(\Delta_n)$, it suffices to show that $D_{\alpha_{n,i}}(\tnu_n) > 0$ for each $i$ and $\omega$-a.e. $n$.  We will show something much stronger, namely that for each $i=1, \dots, \bold{A}(\bmu)$,

$$\frac{D_{\alpha_{n,i}}(\tnu_n)}{s_n} \underset{\omega}{\longrightarrow} \infty$$

Suppose, for a contradiction, that this is not true for some $i$.  Then

$$\frac{d_{\AM(S)}(\tmu_n,\tnu_n)}{s_n} \geq \frac{\frac{1}{K} \cdot d_{\HHH_{\alpha_{n,i}}}(\tmu_n,\tnu_n)-C}{s_n} \geq \frac{\frac{1}{K} \cdot \left|D_{\alpha_{n,i}}(\tmu_n) - D_{\alpha_{n,i}}(\tnu_n)\right|-C}{s_n} \underset{\omega}{\longrightarrow} \infty$$
where $K,C>0$ are the constants from Theorem \ref{distance}.  This implies that $\bnu \notin \AM_{\omega}(S)$, which is a contradiction.
\end{proof}

A \emph{global cut point} of a space $X$ is any point $x \in X$ for which $X\hspace{.025in} \setminus \hspace{.025in} x$ has multiple connected components.  Products of connected spaces are free of cut points which is invariant under homeomorphism.  As a consequence of the product structure of thin cones, we immediately obtain the following corollary:  

\begin{corollary}[Thin cones have no global cut points] \label{no cut points}
Let $\bmu \subset \AM(S)$ be thin.  Then the asymptotic cone of $(\AM(S),\bmu,s_n)$ has no global cut points.
\end{corollary}

In the next section, we prove that thick asymptotic cones consist entirely of cut points.  It will follow that thick and thin asymptotic cones are not homeomorphic; see Corollary \ref{not homeo}.\\

\section{$\RR$-trees in thick cones}

The goal of this section is to prove that each point $\bmu \in \AM_{\omega}(S)$ of a thick cone defines an $\RR$-tree to which there exists a locally constant retraction, Theorem \ref{Beh rtree} below.  These $\RR$-trees encode the uniformly thick directions emanating from $\bmu$.  Our approach uses the hierarchy machinery, following on the work in Behrstock \cite{Beh06} and Behrstock-Minsky \cite{BM08}.\\

One consequence of Theorem \ref{Beh rtree} is that every thick cone consists entirely of cut points.  In Subsection \ref{rel hyp and thick} below, we use the fact that to show that the Teichm\"uller metric is not strongly relatively hyperbolic.
 
\subsection{Sets of sublinear growth} \label{r:sublinear growth section}

For the rest of the paper, fix a thick base sequence $\bmu_0 \subset \AM(S)$ and its corresponding thick asymptotic cone, $\AM_{\omega}(S) = \underset{\omega}{\lim} \left(\AM(S), \bmu_0, s_n\right)$.\\

For any $\bmu \in \AM_{\omega}(S)$, the \emph{set of sublinear growth set of $\bmu$} is:

$$\Fmu = \left\{\bnu \hspace{.05in} \Big| \hspace{.05in} \underset{\omega}{\lim} \underset{Y \subsetneq S}{\sup} \frac{d_{\AM(Y)}(\tmu_n,\tnu_n)}{s_n} = 0\right\}$$

The following theorem mirrors Theorem 6.5 of \cite{Beh06}:

\begin{theorem}\label{Beh rtree}
Any two points $\bx, \by \in \Fmu$ are connected by a unique embedded path in $\AM_{\omega}(S)$ and this path lies in $\Fmu$.  In particular, $\Fmu$ is an $\RR$-tree.
\end{theorem}

Since Behrstock's arguments rely heavily on the distance formula and the hierarchy machinery, much of his proof of Theorem 6.5 passes through to our setting unchanged.  For the sake of completeness, we give a sketch of the argument, detailing adaptations where necessary.

\begin{proof}[Proof of Theorem \ref{Beh rtree}]

Let $\bmu \in \AM_{\omega}(S)$ and let $\bx, \by \in \Fmu$.  For each $n$, let $[\tx_n, \ty_n]$ be an augmented hierarchy path between $\tx_n$ and $\ty_n$ based on a hierarchy $H_n$.   Let $[\bx,\by]$ denote the bi-Lipschitz path in $\AM_{\omega}(S)$ which is the ultralimit of $[\tx_n,\ty_n]$.\\

We first observe that $[\bx,\by] \subset \Fmu$.  The definition of $\Fmu$ and the triangle inequality implies that 

\begin{equation} \label{eqn 1}
\underset{\omega}{\lim} \sup_{Y\subsetneq S} \frac{d_{\AM(Y)}(\tx_n,\ty_n)}{s_n} = 0
\end{equation}

Thus $\Fmu = \mathbf{F}_{\bx} = \mathbf{F}_{\by}$.  Using Behrstock's argument, it follows easily from Theorem \ref{distance} and Lemmas \ref{dist rel to endpoints} and \ref{lip proj} that $[\bx,\by] \subset \Fmu$.\\

We now define a map $\Phi:\AM_{\omega}(S) \rightarrow [\bx,\by]$.  The goal is to prove that this map is a locally constant retraction, from which it follows that $\bx$ and $\by$ cannot be connected in the complement of $[\bx,\by]$ by any path.  This implies that any embedded path between $\bx$ and $\by$ must coincide with $[\bx,\by]$, proving the theorem.\\

Fix any $\bnu \in \AM_{\omega}(S)$.  For each $n$, let $g_{H_n} \subset \CC(S)$ denote the main geodesic of $H_n$.  Let $C_{\tnu_n} \subset g_{H_n}$ be the closest point projection of $\tnu_n$ to $g_{H_n}$.  Since $\CC(S)$ is uniformly hyperbolic, $diam_{\CC(S)}(C_{\tnu_n}) < K$ for some $K$.  Define $\widehat{\Phi}(\tnu_n)$ to be the set of all augmented markings in $[\tx_n,\ty_n]$ whose bases contain any curve in $C_{\tnu_n}$.\\

Set $\Phi(\bnu) = \underset{\omega}{\lim} \hspace{.05in} \widehat{\Phi}(\tnu_n)$.  It follows from Theorem \ref{distance}, Lemma \ref{dist rel to endpoints}, and Equation \ref{eqn 1} that

$$\underset{\omega}{\lim} \frac{diam_{\AM(S)}(\widehat{\Phi}(\tnu_n))}{s_n} = 0$$

Thus $\Phi$ is well-defined.  The majority of the proof involves showing that $\Phi$ is locally constant off of $[\bx,\by]$.\\

Suppose that $\bnu \in \AM_{\omega}(S) \hspace{.05in} \setminus \hspace{.05in} [\bx,\by]$.  We prove that there is a sequence $c_n>0$ depending only on $S$ such that if $\bnu' \in \AM_{\omega}(S)$ satisfies $d_{\AM(S)}(\tnu_n, \tnu'_n) < c_n \cdot d_{\AM(S)}(\tnu_n, \Phi(\tnu_n))$, then $\Phi(\bnu) = \Phi(\bnu')$.\\

For each $n$, let $\Tgamma_n$ be an augmented hierarchy path between any point in $\HPhi(\tnu_n)$ and $\tnu_n$ based on a hierarchy $G_n$.  Since $\Phi$ is well-defined, our choice of point in $\HPhi(\tnu_n)$ does not matter.  Let $\Tlambda_n$ be an augmented hierarchy path between $\tnu_n$ and $\tnu'_n$ based on a hierarchy $L_n$.\\  

For each $n$, let $D_n = \underset{Y \subsetneq S}{\sup} d_{\AM(Y)}(\tx_n,\ty_n)$.  Let $K_n$ be the quasiisometry constant coming from Theorem \ref{distance} for a threshold of $t_n>4 \cdot M + D_n$, where $M$ is the constant from Lemma \ref{beh ineq}.  Note that $\underset{\omega}{\lim} \frac{K_n}{s_n} = 0$.\\

Following on \cite{Beh06}, we break the proof into two cases:\\

\paragraph{\textbf{Case (1)}} For $\omega$-a.e. $n$, we have $|g_{G_n}| > \frac{1}{K_n} d_{\AM(S)}(\tnu_n, \HPhi(\tnu_n)) + K_n$.\\

That is, the distance in $\AM(S)$ between $\tnu_n$ and $\HPhi(\tnu_n)$ is coarsely dominated by curve complex distance.  Using Theorem \ref{distance} and hyperbolicity of $\CC(S)$, Behrstock shows that one can choose $c_n>0$ small enough so that for $\omega$-a.e. $n$, there is a ball in $\CC(S)$ containing both $\tnu_n$ and $\tnu'_n$ which is disjoint from $g_{H_n}$, implying $\Phi(\bnu) = \Phi(\bnu')$ by definition.  This works in $\AM(S)$ as well.\\

\paragraph{\textbf{Case (2)}} For $\omega$-a.e. $n$, we have $|g_{G_n}| \leq \frac{1}{K_n} d_{\AM(S)}(\tnu_n, \HPhi(\tnu_n)) + K_n$.\\

In this case, the choice of $K_n$ implies that for $\omega$-a.e. $n$, the hierarchy $G_n$ will have some domain $Y_n \subsetneq S$ for which $d_{Y_n}(\tnu_n, \HPhi(\tnu_n)) > 4M + D_n$.  The idea of Behrstock's proof in the case of $\MM(S)$ goes through to $\AM(S)$ without trouble: If we can show that any augmented hierarchy path from $\tnu'_n$ to any point on $[\tx_n,\ty_n]$ passes through one of the $Y_n$, then it follows from the definition of $\Fmu$ and Theorem \ref{bgit} that the base geoedesic of such a path must intersect $g_{G_n}$ at $\partial Y_n$, with hyperbolicity of $\CC(S)$ and the definition of $\Fmu$ implying that $\Phi(\bnu) = \Phi(\bnu')$.  In order for this not to happen, $\tnu'_n$ and $\HPhi_n(\tnu_n)$ would have to be close in $\CC(Y_n)$ for each such $Y_n$, making $d_{Y_n}(\tnu_n,\tnu'_n) \asymp d_{Y_n}(\tnu_n,\HPhi(\tnu_n))$.  We then find a finite collection of $(4M + D_n)$-large links for $G_n$ whose lengths bound $|\Tgamma_n|$ from above and which (at worst) covers $\Tgamma_n$ by a bounded degree.  By making some estimates using Theorem \ref{distance}, Lemma \ref{beh ineq}, and general properties of hierarchies, we can then bound $d_{\AM(S)}(\tnu_n, \HPhi(\tnu_n))$ from above as a fraction of $|\Tlambda_n|$.  Since $\Tlambda_n$ is a uniform quasigeodesic, choosing $c_n$ sufficiently small produces a contradiction.  We now sketch both how to obtain the finite collection of subsurfaces and the estimates mentioned above.\\

Let $Y \subsetneq S$ be a $(4M+D_n)$-large link for the hierarchy $G_n$ between $\tnu_n$ and $\HPhi(\tnu_n)$, i.e. $Y \in  \mathcal{L}_{4M+D_n}(\tnu_n,\HPhi(\tnu_n))$.  Suppose also that $Y$ is also a domain in $L_n$.\\

For any $Y \subsetneq S$, recall our definition of active segment, $\Tlambda_{n,Y} \subset \Tlambda_n$.  For us, these will play the role of the $J_{i,Y}$ from \cite{Beh06}[p. 1570], as Lemma \ref{r:active segment} shows they have the required properties of the $J_{i,Y}$.\\

We say that $\Tlambda_n$ \emph{has traversed $Y$} if $d_Y(\tnu'_n, \HPhi(\tnu_n)) < M$.  This notion is measuring the progress that $\Tlambda_n$ makes along $Y$ relative to how much progress $\Tgamma_n$ must make along $Y$.\\

Let $Y_n \in \mathcal{L}_{4M + D_n}(\tnu_n,\HPhi(\tnu_n))$ be such that if $Z \in \mathcal{L}_{4M + D_n}(\tnu_n,\HPhi(\tnu_n))$ and $Y_n \pitchfork Z$, then $Y_n \prec_t Z$.  We call $Y_n$ an \emph{initial domain} of $G_n$ and the collection of initial domains are the first $(4M+D_n)$-large links through which any augmented hierarchy path based on $G_n$ passes.  A simple topological count proves that the set of initial domains has cardinality at most $\xi(S)$.  The set 

\[\mathcal{I}_n = \{Y_n\} \cup \{Z \in \mathcal{L}_{4M+D_n}(\tnu_n,\HPhi(\tnu_n)) | Y_n \prec_t Z\}\]
collects all the subsurfaces whose active segments necessarily follow the active segment for $Y_n$ along $\Tgamma_n$.   Note that $\mathcal{I}_n$ is nonempty by the assumption.  Since the proof of \cite{Beh06}[Lemma 6.6] is easily seen to hold in our setting, the active segments of the subsurfaces in $\mathcal{I}_n$ cover $\Tgamma_n$ with degree bounded by $2\xi(S)$.\\

Using Behrstock's argument and Lemma \ref{r:active segment}, there is an $\alpha_n>0$ depending only on $S$ such that

\[\sum_{Z \in \mathcal{I}_n} d_Z (\tnu_n, \HPhi(\tnu_n)) > \alpha_n d_{\AM(S)} (\tnu_n, \HPhi(\tnu_n)) \indent \text{ and } \indent 2k\xi(S) \geq \sum_{Z \in \mathcal{I}_n} |\Tgamma_{n,Z}|\]
for $\omega$-a.e. $n$, where $\alpha$ depends only on $S$ and $d_Z = \HHH_{\gamma}$ when $Z$ is an annulus with core $\gamma$.\\

If $\Tlambda_n$ has traversed $Y$, then Lemma \ref{r:active segment} implies that there is a uniform $\beta>0$ such that $|\Tlambda_{n,Y}|> \beta \cdot d_{Y}(\tnu_n,\HPhi(\tnu_n))$.  Thus we have

$$2|\Tlambda_n| \xi(S) \geq \sum_{Z \in \mathcal{I}_n} |\Tgamma_{n,Z}| \geq \sum_{Z \in \mathcal{I}_n} \beta d_{Z}(\tnu_n, \HPhi(\tnu_n)) \geq \alpha_n \beta d_{\AM(S)}(\tnu_n, \HPhi(\tnu_n))$$
implying that $|\Tlambda_n| \geq \frac{\alpha_n \beta d_{\AM(S)}(\tnu_n,\HPhi(\tnu_n))}{2\xi(S)}$.  Choosing $c_n \leq \frac{\alpha_n \beta}{2\xi(S)}$ completes the proof.
\end{proof}

The following is a slight enhancement of Theorem \ref{Beh rtree}, mirroring Theorem 3.4 of \cite{BM08}:

\begin{theorem}\label{rtree}
Given $\bmu \in \AM_{\omega}(S)$, there is a continuous map:

$$\wp= \wp_{\bmu}: \AM_{\omega}(S) \rightarrow \Fmu$$

with the following properties:

\begin{enumerate}
\item $\wp$ is the identity on $\Fmu$.
\item $\wp$ is locally constant in $\AM_{\omega}(S) \setminus \Fmu$.
\end{enumerate}

\end{theorem}

Since we use the map $\wp$ extensively in the rest of the paper, we give the short proof from \cite{BM08}, half of which is defining $\wp$.
 
\begin{proof}
Fix $\bx \in \AM_{\omega}(S)$ and let $\gamma$ be any path from $\bx$ to $\Fmu$ .  Let $\by$ be the entry point of $\gamma$ into $\Fmu$.  Then $\by$ is independent of the choice of $\gamma$, for if $\bz$ is another point of entry for another path $\gamma'$ from $\bx$ to $\Fmu$, then $\gamma \cup \gamma'$ is a path between points of $\Fmu$ which lies outside of $\Fmu$, contradicting Theorem \ref{Beh rtree}.\\

Define $\wp(\bx) = \by$.  We clearly have that $\wp$ restricts to the identity on $\Fmu$.  Since $\AM_{\omega}(S)$ is locally path connected, $\wp$ is locally constant on $\AM_{\omega}(S)\setminus \Fmu$, for any point in a small ball $U$ around $\bx$ can first be connected to $\bx$ by a path lying in $U$.\\

Continuity of $\wp$ is immediate from the definition and the fact that $\AM_{\omega}(S)$ is locally path connected.\end{proof}

As noted in the previous section, we have the following immediate corollary of Theorem \ref{rtree}:

\begin{corollary} \label{not homeo}
Every point of a thick asymptotic cone is a cut point.  Thus thick asymptotic cones are not homeomorphic to thin asymptotic cones.
\end{corollary}

Before closing this section, we collect a few easy but important properties of $\wp$:

\begin{lemma}\label{prop of wp}
For any $\bmu \in \AM_{\omega}(S)$, we have:
\begin{enumerate}
\item For any $\bnu \in \Fmu$, we have $\wp_{\bmu} \equiv \wp_{\bnu}$. \label{wp 1}
\item  If $\wp(\bx) \neq \bmu$ for some $\bx \in \AM_{\omega}(S)$, then $\tx_n$ and $\tmu_n$ \emph{separate in} $\CC(S)$, that is $d_S(\tx_n, \tmu_n) \underset{\omega}{\longrightarrow} \infty$.\label{wp 2}
\item If $\wp(\bx) \neq \wp(\by)$, then $d_S(\tx_n, \ty_n)  \underset{\omega}{\longrightarrow} \infty$. \label{wp 3}
\end{enumerate}
\end{lemma}

\begin{proof}

Property (\ref{wp 1}) follows from the fact that $\Fmu = \Fnu$, which was shown in the proof of Theorem \ref{Beh rtree}.  Property (\ref{wp 3}) follows from Property (\ref{wp 2}) and the triangle inequality.\\

To see Property (\ref{wp 2}), let $\bnu = \wp(\bx)\in \Fmu$.  We claim that $d_{S}(\tnu_n, \tmu_n) \underset{\omega}{\longrightarrow} \infty$.  If not, then there is some $N>0$ such that $d_S(\tnu_n, \tmu_n) < N$ for $\omega$-a.e. $n$.  Since $\bnu \neq \bmu$, we can use a threshold bigger than $N$ in the distance formula (Theorem \ref{distance}) to conclude that $\tnu_n$ and $\tmu_n$ have a large link $Y_n$ with $d_{Y_n}(\tnu_n,\tmu_n) > K_n$ where $\underset{\omega}{\lim}\frac{K_n}{s_n} >0$, which violates the definition of $\Fmu$.\\

Let $\Gamma_n$ be any augmented hierarchy path between $\tx_n$ and $\tmu_n$ with base geodesic $g_n\subset \CC(S)$.  Since $\Gamma_n$ monotonically shadows $g_n$, the definition of $\wp_{\bmu}$ implies that $|g_n| \succ d_{S}(\tnu_n, \tmu_n)$ and Lemma \ref{large link} implies that $|g_n| \asymp d_S(\tx_n, \tmu_n)$.  Thus Property (\ref{wp 2}) holds.

\end{proof}

\subsection{Thickness and nonrelative hyperbolicity of $\TT(S)$} \label{rel hyp and thick}

We make a brief stop to observe a consequence of Theorem \ref{Beh rtree}, namely that the Teichm\"uller metric is not strongly relatively hyperbolic.\\

The notion of a relatively hyperbolic space is aimed to capture the geometry of spaces which behave like hyperbolic spaces outside of some controlled collection of subspaces, called the peripheral subspaces.  In \cite{MM99}, Masur-Minsky showed that the electrified Teichm\"uller space---that is, $(\TT(S),d_T)$ with all of its thin parts coned off--- is quasiisometric to $\CC(S)$, showing that $(\TT(S),d_T)$ is \emph{weakly relatively hyperbolic}.  The now-standard notion of relative hyperbolicity, sometimes called \emph{strong relative hyperbolicity}, includes Farb's bounded coset penetration property \cite{Farb98}, which, roughly speaking, requires that two geodesics with nearby endpoints outside of the peripheral subspaces must interact with the peripheral subspaces in essentially the same way.  See \cite{Raf14} for related properties of Teichm\"uller geodesics.\\

In \cite{BDM08}, Behrstock-Drutu-Mosher introduced \emph{thickness}, an inductively-defined property which measures how far away from being strongly relatively hyperbolic a group or space is.  We use their construction to show that $(\TT(S),d_T)$ is not strongly relatively hyperbolic.  We now briefly review their construction.

\begin{definition}[Networks of subspaces]
Let $X$ be a metric space, $\mathcal{Y}$ a collection of subspaces of $X$, and $\tau \geq 0$.  We say that $\mathcal{Y}$ is a \emph{$\tau$-network of subspaces of $X$} if 
\begin{itemize}
\item[$(N_1)$] $X = \bigcup_{Y \in \mathcal{Y}} \mathcal{N}_{\tau}(Y)$
\item[$(N_2)$] Any two elements $Y, Y' \in \mathcal{Y}$ can be \emph{thickly connected}: There exists $Y = Y_1, Y_2, \dots Y_n = Y'$ with $Y_i \in \mathcal{Y}$ and $\mathrm{diam}(\mathcal{N}_{\tau}(Y_i) \cap \mathcal{N}_{\tau}(Y_{i+1})) = \infty$ for each $1\leq i< n$.
\end{itemize}
\end{definition}

A metric space $X$ is called \emph{unconstricted} if, for some choice of ultrafilter and scaling sequence, every asymptotic cone of $X$ has no cut points.

\begin{definition}[Thickness]
Let $X$ be a metric space and $\mathcal{Y}$ a collection of subspaces.  We say $X$ is \emph{thick of order 1} if the following hold:
\begin{enumerate}
\item[$(T_1)$] $X$ is not unconstricted
\item[$(T_2)$] For some $\tau\geq 0$, $\mathcal{Y}$ is a $\tau$-network of subspaces for $X$ and every $Y \in \mathcal{Y}$ is unconstricted when endowed with the restricted metric of $X$.
\end{enumerate}
\end{definition}
Behrstock-Drutu-Mosher proved that thickness is a quasiisometry invariant (\cite{BDM08}[Remark 7.2]) and that thick groups and spaces are not strongly relatively hyperbolic (\cite{BDM08}[Corollary 7.9]).  As an application, they showed that $\MCG(S)$ is thick for $\xi(S)>1$.  The following is a corollary of Theorem \ref{Beh rtree}:

\begin{corollary} \label{r:thick}
For $\xi(S)>1$, $(\TT(S),d_T)$ is thick of order 1 and thus not strongly relatively hyperbolic.
\end{corollary}

\begin{proof}
It suffices to prove the theorem for $\AM(S)$.  Since thick cones consist entirely of cut points by Corollary \ref{not homeo}, $\AM(S)$ is not unconstricted.\\

Let $\mathcal{Y}= \{Q(\alpha)| \alpha \in \CC(S)\}$.  Note that $\AM(S) = \bigcup_{\alpha \in \CC(S)} Q(\alpha)$ and $Q(\alpha) \cap Q(\beta) = Q(\alpha \cup \beta)$ if $\alpha \cap \beta = \emptyset$, so $\mathcal{Y}$ is a 0-network of subspaces of $\AM(S)$ by connectivity of $\CC(S)$.  Moreover, each element of $\mathcal{Y}$ is  quasiisometric to a product space by Lemma \ref{prod reg}, so $\mathcal{Y}$ consists of unconstricted spaces.  Thus $\AM(S)$ is thick of order 1 by definition.
\end{proof}

\section{The Rank Theorem}

In this section, we prove the Dimension Theorem \ref{r:main intro}; the Rank Theorem \ref{r:rank intro} is an immediate corollary.  Our proof follows the general outline of \cite{BM08}.  Much of Behrstock-Minsky's approach goes through to $\AM(S)$ without issue.  As we did in the proof of Theorem \ref{Beh rtree}, we give sketches of proofs, adding details when some adaptation is necessary.  The proof of Theorem \ref{dim thm 2} has the following general outline.\\

Using results from dimension theory, the proof of the Rank Theorem is reduced to the following statement (see Subsection \ref{rank thick sect}): There is a family of subspaces $\mathcal{L}$ of $\AM_{\omega}(S)$ such that $\text{dim } L< \xi(S)$ for each $L \in \mathcal{L}$ and any two points in $\AM_{\omega}(S)$ can be separated by some $L \in \mathcal{L}$.\\

In Theorem \ref{global} of Subsection \ref{sep sect}, we use the product structures described in Subsection \ref{sep prod sect} and the locally-constant retractions constructed in Theorem \ref{rtree} to define nice retractions onto $\RR$-trees living in various $\AM(\bW)$, where $\bW$ is a sequence of proper subsurfaces arising as components of these product structures.\\

An easy inductive argument shows that any two sequences in $\AM_{\omega}(S)$ must diverge linearly in some subsurface curve complex.  In Theorem \ref{sep thm} of Subsection \ref{sep sect}, we use this fact, the hierarchy machinery, and the aforementioned retractions to the subsurface $\RR$-trees of Theorem \ref{global} to build the separators we want.

\subsection{Separating product regions}\label{sep prod sect}

For any sequence of subsurfaces $\bW = \{W_n\}$ and any $\bmu \in \AM_{\omega}(S)$, set

$$\Fwmu = \{\by \in \AM_{\omega}(\bW) | \lim_{\omega} \sup_{Z} \frac{1}{s_n} d_{\AM(Z)}(\tmu_n, \ty_n) = 0\}$$

Since $\bmu$ is a thick sequence, $\pi_{\AM(Y)}(\tmu_n)$ is a thick sequence in $\AM(Y)$ for any subsurface $Y \subset S$.\\

Theorems \ref{Beh rtree} and \ref{rtree} obviously hold with $S$ and $\Fmu$ replaced by $\bW$ and $\Fwmu$.\\

Recall from Proposition \ref{prod reg} of Subsection \ref{prod sect} that $Q(\partial Y) \approx \prod_{Z \in \sigma(\partial Y)} \AM(Z)$ for any subsurface $Y \subset S$.  It follows that for any sequence $\bW$ and $\bmu \in \AM_{\omega}(S)$, there is a naturally defined subset $Q_{\omega}(\partial \bW) \subset \AM_{\omega}(S)$ with the property that

$$Q_{\omega}(\partial \bW) \approx \prod_{Y \in \sigma(\partial \bW)} \AM(Y)$$
where $\AM(\bW)$ is one of the components since $\bW \in \sigma(\partial \bW)$ and the identification is a bi-Lipschitz homeomorphism.  By Lemma \ref{dist to prod}, given any $\bmu \in \AM_{\omega}(S)$, the distance from $\bmu$ to $Q_{\omega}(\partial \bW)$ is approximately

$$\rho(\bmu, \partial \bW) \asymp \lim_{\omega} \frac{1}{s_n} \sum_{Y \pitchfork W_n} \left[\left[d_{Y}(\tmu_n, \partial W_n)\right]\right]_K$$

Define $P_{\bW, \bmu} \subset Q_{\omega}(\partial \bW)$ to be all points whose $\AM_{\omega}(\bW)$ coordinate lies in $\Fwmu$.  Since the quasiisometry in Proposition \ref{prod reg} is defined via subsurface projections, we have the following characterization of $P_{\bW,\bmu}$:

\begin{lemma} \label{prop of P}
$P_{\bW,\bmu}$ is the set of points $\bx \in Q_{\omega}(\partial \bW)$ such that 

\begin{enumerate}
\item $\pi_{\AM_{\omega}(\bW)}(\bx) \in \Fwmu$
\item $\rho(\bx, \partial \bW) = 0$
\end{enumerate}
\end{lemma}

For each $n$, set $W_n^c = \sigma(\partial W_n) \setminus W_n$ and $\bW^c = \{W_n^c\}$.  Then the asymptotic cone $\AM_{\omega}(\bW^c)$ has the product decomposition by Lemma \ref{prod reg}

$$\AM_{\omega}(\bW^c) \approx \prod_{Y \in \sigma(\partial \bW) \setminus \bW} \AM_{\omega}(Y)$$

Thus we have:

\begin{lemma}\label{P is prod}
There is a bi-Lipschitz identification

$$P_{\bW,\bmu} \approx \Fwmu \times \AM_{\omega}(\bW^c)$$
\end{lemma}

The following theorem, mirroring Theorem 3.4 of \cite{BM08}, constructs locally constant projections to the subsurface $\RR$-trees,  $\Fwmu$:

\begin{theorem} \label{global}
For any $\bmu \in \AM_{\omega}(\bW)$, there is a continuous map

$$\Phi = \Phi_{\bW, \bmu}: \AM_{\omega}(S) \rightarrow \Fwmu$$
with the following properties:

\begin{enumerate}
\item $\Phi$ restricted to $\Pwmu$ is projection to the first factor of $\Pwmu \approx \Fwmu \times \AM_{\omega}(\bW^c)$. \label{prop 1}
\item $\Phi$ is locally constant on $\AM_{\omega}(S) \setminus \Pwmu$. 
\end{enumerate}
\end{theorem}

\begin{proof}

For any $\bx \in \AM_{\omega}(S)$, set $\Phi(\bx) = \wp_{\bW,\bmu} \circ \pi_{\AM_{\omega}(\bW)}(\bx)$.\\

Property (\ref{prop 1}) follows immediately from the definition of $\Phi$ and Lemma \ref{P is prod}.\\

Let $\bx, \by \in \AM_{\omega}(S)$ be such that $\Phi(\bx) \neq \Phi(\by)$.  If $\pi_{\AM_{\omega}(\bW)}(\bx) \notin \Fwmu$, then the result follows from Theorem \ref{rtree} and the continuity of $\pi_{\AM_{\omega}(\bW)}$.\\ 

The case when $\pi_{\AM_{\omega}(\bW)}(\bx) \in \Fwmu$ is the bulk of the proof.  Note that Lemma \ref{prop of P} implies that $\rho(\bx, \partial \bW) >0$.  Since $\pi_{\AM_{\omega}(\bW)}(\bx) \in \Fwmu$, we have $d(\bx, \Phi(\bx)) = \rho(\bx, \partial \bW)$ by definition of $\Phi$ and Lemma \ref{lip proj}.  The goal of the proof is to find a lower bound for $d_{\AM_{\omega}(S)}(\bx,\by)$ in terms of $\rho(\bx, \partial \bW)$.\\

The main idea is the following: Since $\Phi(\bx) \neq \Phi(\by)$, Proposition \ref{dist to prod} tells us that $\underset{\omega}{\lim}\frac{1}{s_n} d_{Y_n}(\tx_n,\ty_n)>0$ for some subsurfaces $Y_n \pitchfork \partial W_n$.   Bounding these distances below in terms of $d(\bx, \Phi(\bx))$ is the key, which is done via Behrstock's inequality (Lemma \ref{beh ineq}).  A threshold trick with the distance formula (Theorem \ref{distance}) then gives the desired bound.\\

First, observe that $\rho(\bx, \partial \bW)>0$ implies that 

\begin{equation}\label{phi 1}
\underset{\omega}{\lim} \frac{1}{s_n} \sum_{Y \pitchfork \partial W_n} \left[\left[(\tx_n, \partial W_n)\right]\right]_K > c > 0
\end{equation}
and thus there are subsurfaces $Y_n \subset S$ for which $\underset{\omega}{\lim} \frac{1}{s_n} d_{Y_n}(\tx_n, \partial W_n) >0$.  Next, observe that since $d_{W_n}(\ty_n, \pi_{\AM(W_n)}(\ty_n)) \asymp 1$ for all $n$, the definition of $\Phi$ and Lemma \ref{prop of wp} imply that $d_{W_n}(\tx_n,\ty_n) \underset{\omega}{\longrightarrow} \infty$.\\

We can now use Behrstock's inequality (Proposition \ref{beh ineq}) to build a bound for $d_{Y_n}(\tx_n, \ty_n)$.  For any such $Y_n \pitchfork W_n$ as above, we have $d_{Y_n}(\tx_n, \partial W_n) >M_1$ for $\omega$-a.e. $n$, so that Proposition \ref{beh ineq} implies that $d_{W_n}(\tx_n, \partial Y_n)< M_1$ for $\omega$-a.e. $n$.  Since $d_{W_n}(\tx_n, \ty_n) \underset{\omega}{\longrightarrow} \infty$, the triangle inequality implies that $d_{W_n}(\ty_n, \partial Y_n)>d_{W_n}(\tx_n, \ty_n) - M_1 - D$ for $\omega$-a.e. $n$ where $D$ is the Lipschitz constant for subsurface projections from Lemma \ref{lip proj}.  Again using the fact that $d_{W_n}(\tx_n, \ty_n) \underset{\omega}{\longrightarrow} \infty$, we may apply Proposition \ref{beh ineq} to obtain $d_{Y_n}(\ty_n, \partial W_n)<M_1$, with the triangle inequality implying that $d_{Y_n}(\tx_n, \ty_n) > d_{Y_n}(\tx_n, \partial W_n) - M_1 - D$.\\

Applying this estimate to each $Y$ appearing in equation (\ref{phi 1}), one can then use an easy trick with the thresholds in the distance formula (Theorem \ref{distance}) to obtain the desired bound $d_{\AM_{\omega}(S)}(\bx, \by)>c'>0$.

\end{proof}

\subsection{Separators in $\AM_{\omega}(S)$} \label{sep sect}

In this subsection, we construct a family of subspaces $\mathcal{L}$ of $\AM_{\omega}(S)$ which separates points.  It is significant that these separators will be homeomorphic to $\AM_{\omega}(W)$, where $W \subset S$ is some proper, essential subsurface.  In the next subsection, we use an argument from \cite{BM08} to reduce the Dimension Theorem to the existence of such a family.

\begin{theorem} \label{sep thm}
There exists a family, $\mathcal{L}$, of subspaces of $\AM_{\omega}(S)$ which separates points and for each $L \in \mathcal{L}$, $L$ is isometric to $\AM_{\omega}(W)$, where $W \subset S$ is some essential (not necessarily connected) subsurface with $r(W) \leq r(S)$. 
\end{theorem}

\begin{proof}

The proof proceeds exactly as the proof of Theorem 3.6 of \cite{BM08}.\\

Fix $\bx \neq \by \in \AM_{\omega}(S)$.  We claim there is a sequence of subsurfaces $\bW$ such that

\begin{enumerate}
\item $d_{\AM_{\omega}(\bW)}(\bx, \by) > 0$
\item For any $\mathbf{Y} \subsetneq \bW$, we have $d_{\AM_{\omega}(\mathbf{Y})}(\bx, \by) = 0$.

\end{enumerate}

In the case that $S=\bW$ works, the result follows from Theorem \ref{Beh rtree}.  If not, then we may choose $\bW' \subsetneq \bW$ with $d_{\AM_{\omega}(\bW')}(\bx, \by) > 0$ and proceed.  This process terminates because $r(\bW') < r(\bW)$.\\

Let $\bx' = \pi_{\AM_{\omega}(\bW)}(\bx)$ and $\by' = \pi_{\AM_{\omega}(\bW)}(\by)$, so that $\bx' \neq \by'$ but $\by' \in \Fwxx$ by assumption.  Let $\bz \in \Fwxx$ be any point lying on path between $\bx'$ and $\by'$ in $\Fwxx$.  Theorem \ref{Beh rtree} implies that $\bz$ separates $\bx'$ and $\by'$ in $\AM_{\omega}(\bW)$.\\

Let $L \subset \Pwxx$ be the subspace bi-Lipschitz homeomorphic to $L \approx \{\bz\} \times \AM_{\omega}(\bW^c)$ by Lemma \ref{P is prod}.  We clearly have that $L$ separates $\Pwxx$.  We now use the locally constant retraction $\Phi = \Phi_{\bW, \bx'}: \AM_{\omega}(S) \rightarrow \Fwxx$ constructed in Theorem \ref{global} to show that it separates all of $\AM_{\omega}(S)$ with $\bx$ and $\by$ on different sides.\\

Note that $\Phi(\bx) = \bx'$ and $\Phi(\by) = \by'$.  Separate $\Fwxx \setminus \{\bz\}$ into two different components $E_1$ and $E_2$.  Since $\Phi$ is continuous, $\Phi^{-1}(E_1)$ and $\Phi^{-1}(E_2)$ are disjoint open sets containing $\bx$ and $\by$, respectively.  Since $\Phi$ is locally constant, we see that $\Phi^{-1}(\bz)= L \cup V$, where $V$ is an open set disjoint from $E_1$ and $E_2$.  Thus $L$ separates $\bx$ from $\by$ in $\AM_{\omega}(S)$.\\

Since $L$ is homeomorphic to the the asymptotic cone $\AM_{\omega}(\bW^c)$, $L$ is closed.  Since $\bW$ is $\omega$-a.e. homeomorphic to some subsurface $W \subset S$, we have that $L$ is isometric to $\AM_{\omega}(W^c)$, completing the proof.
\end{proof}

\subsection{Proof of the Dimension and Rank Theorems}\label{rank thick sect}

The finishing touches to the proofs of the Dimension and Rank Theorems for $\TT(S)$ proceed identically to that of $\MCG(S)$ as in \cite{BM08}, with only a small note to deal with the difference between thick and thin cones.  We include the details for completeness.\\

Dimension theory is a branch of topology which studies various notions of dimension, the main three being small inductive dimension ($\ind$), large inductive dimension ($\Ind$), and covering dimension ($\dim$), which is also called the topological dimension (see \cite{Eng95}).  These also have inductive versions, where one takes the supremum of the above dimensions over locally compact subspaces of the ambient space, which we denote by $\wind$, $\wInd$, and $\wdim$.\\\
 
Our main theorem is:

\begin{theorem} [Dimension for thick cones] \label{dim thm 2}
For any thick base sequence, $\wind(\AM_{\omega}(S)) = \wInd(\AM_{\omega}(S)) = \wdim(\AM_{\omega}(S)) = r(S)$
\end{theorem}

Since $\ind$ is subadditive over products (see \cite{Eng95}), we have the following immediate corollary of Theorems \ref{thin cones} and Theorem \ref{dim thm 2}:

\begin{corollary} [Dimension for thin cones] \label{dim thm 3}
For any thin base sequence, $\wind(\AM_{\omega}(S)) = \wInd(\AM_{\omega}(S)) = \wdim(\AM_{\omega}(S)) = r(S)$
\end{corollary}

The Dimension Theorem \ref{r:main intro} follows immediately from Theorem \ref{dim thm 2} and Corollary \ref{dim thm 3} because $\RR^n$ is locally compact.  The Rank Theorem \ref{r:rank intro} is an immediate consequence.\\

Before preceding with the proof, we recall two lemmata from \cite{BM08}, the first of which unifies the above notions of dimension in our setting:

\begin{lemma}[Lemma 4.1 in \cite{BM08}]
For a metric space, $\wind = \wInd = \wdim$.
\end{lemma} 

The second lemma reduces gives the characterization of small inductive dimension to a problem of producing separators:

\begin{lemma}[Lemma 4.2 in \cite{BM08}] \label{sep lemma}
Suppose a metric space $X$ admits a family of subspaces $\mathcal{L}$ which separates points and $\wind(L)\leq N-1$ for each $L \in \mathcal{L}$.  Then $\wind(X) =\wInd(X) = \wdim(X) \leq N$.
\end{lemma}

We now proof Theorem \ref{dim thm 2} as an easy application of Lemma \ref{sep lemma} with $N = r(S)$ via Theorem \ref{sep thm} :

\begin{proof}[Proof of Theorem \ref{dim thm 2}]
First, note that $\wind(\AM_{\omega}(S)) \geq r(S)$ follows from Minsky's Product Regions Theorem \ref{r:min prod}, as a product of rays in horoballs gives an $r(S)$-dimensional orthant in $\TT(S)$.\\

When $r(S) = 1$, then $S = S_{0,4}, S_{1,1},$ or $S_{0,2}$.  In the first two cases, $\TT(S)$ is $\HH^2$, and in the latter case $\AM(S)$ is a horoball, all of which are hyperbolic and thus have $\RR$-trees as asymptotic cones.  As $\RR$-trees are well known to satisfy $\wind = 1$, the theorem holds in these cases.\\

In the higher complexity cases, Theorem \ref{sep thm} produced the family $\mathcal{L}$, which separates points in $\AM_{\omega}(S)$ and consists of subspaces, $L$ homeomorphic to $\AM_{\omega}(W^c)$, where $W$ is some essential subsurface.  The function $r$ is additive over products, so $r(W^c) \leq r(S) - 1$.  Applying induction and the fact that $\ind$ is subadditive over products and additive over disjoint unions (see \cite{Eng95}) gives that $\wind(L) \leq r(S) - 1$.  Lemma \ref{sep lemma} completes the proof. 
\end{proof}

\end{document}